\documentclass[a4paper]{article}%

\usepackage{graphicx}%
\usepackage[latin1]{inputenc}%
\usepackage[T1]{fontenc}%
\usepackage{amssymb,amsmath,amsthm,amscd,mathrsfs}%
\usepackage{tikz-cd}
\usepackage[all]{xy}%
\usepackage{xcolor}%
\usepackage[margin=3cm]{geometry}%
\usepackage{paralist}%
\usepackage[colorlinks=true, allcolors=blue]{hyperref}%
\hypersetup{%
	colorlinks=true,%
	linkcolor=magenta,%
	filecolor=blue,%
	urlcolor=blue,%
	citecolor=blue,%
    pdftitle={Symplectic Automorphisms Nikulin Orbifolds},%
    pdfpagemode=FullScreen,}%
\usepackage[capitalize, nameinlink]{cleveref}%
\usepackage{colortbl}%
\usepackage[hypcap=false]{caption}%
\usepackage{mathtools}

\newtheorem{thm}{Theorem}[section]%
\newtheorem{defi}[thm]{Definition}%
\newtheorem{prop}[thm]{Proposition}%
\newtheorem{lemme}[thm]{Lemma}%
\newtheorem{cor}[thm]{Corollary}%
\newtheorem{nota}[thm]{Notation}%
\theoremstyle{remark}%
\newtheorem{rmk}[thm]{Remark}%

\newcommand{\Z}{\mathbb{Z}}%
\newcommand{\FF}{\mathbb{F}}%
\newcommand{\ZZ}{\mathbb{Z}}%
\newcommand{\R}{\mathbb{R}}%
\newcommand{\QQ}{\mathbb{Q}}%
\DeclareMathOperator{\Pic}{Pic}%
\DeclareMathOperator{\rk}{rk}%
\DeclareMathOperator{\Fix}{Fix}%
\DeclareMathOperator{\Aut}{Aut}%
\DeclareMathOperator{\Bir}{Bir}%
\DeclareMathOperator{\Vect}{Vect}%
\DeclareMathOperator{\id}{id}%
\DeclareMathOperator{\GL}{GL}%
\DeclareMathOperator{\HH}{H}%
\newcommand{\Q}{\mathbb{Q}}%
\DeclareMathOperator{\C}{\mathbb{C}}%
\DeclareMathOperator{\Ima}{Im}%
\DeclareMathOperator{\Rea}{Re}%
\DeclareMathOperator{\Div}{div}%
\DeclareMathOperator{\Mon}{Mon}%
\DeclareMathOperator{\II}{\textnormal{II}}%
\newcommand{\RR}{\mathbb{R}}%

\newcommand{\eq}[1][r]%
{\ar@<-3pt>@{-}[#1]%
\ar@<-1pt>@{}[#1]|<{}="gauche"%
\ar@<+0pt>@{}[#1]|-{}="milieu"%
\ar@<+1pt>@{}[#1]|>{}="droite"%
\ar@/^2pt/@{-}"gauche";"milieu"%
\ar@/_2pt/@{-}"milieu";"droite"}%

\newcommand{\incl}[1][r]%
  {\ar@<-0.2pc>@{^(-}[#1] \ar@<+0.2pc>@{-}[#1]}%

\begin{document}%
\title{\bf Automorphisms of Nikulin-type orbifolds}%
\author{Simon \textsc{Brandhorst}\\%
  \texttt{brandhorst@math.uni-sb.de}%
  \and%
  Gr{\'e}goire \textsc{Menet}\\%
  \texttt{gregoire.menet@ac-amiens.fr}%
  \and%
  Stevell \textsc{Muller}\\%
  \texttt{muller@math.uni-hannover.de}%
  }%

\maketitle%
\abstract{Nikulin-type orbifolds are certain singular 4-dimensional irreducible holomorphic symplectic varieties.  
We show that the monodromy group of Nikulin-type orbifolds is maximal and classify finite order symplectic automorphisms up to deformation in terms of their action on the second integral cohomology group.}%
\section{Introduction}%
In recent years, (singular) irreducible holomorphic symplectic (IHS) varieties have received increasing attention, in particular due to the generalizations of the Bogomolov decomposition theorem in this setting. It is natural to study their symmetries. In this paper, we provide the first complete classification of symplectic automorphisms of finite order for a family of singular IHS varieties.

The generalization of the global Torelli theorem from smooth to singular IHS varieties (see for instance \cite{Bakker,Ulrike0}) allows one to classify automorphisms in terms of their action on the second integral cohomology group. For a fixed deformation type, these actions can be enumerated using an algorithm developed in \cite{brandhorst-hofmann}. To apply the global Torelli theorem, one needs three deformation invariants: The Beauville--Bogomolov form, the wall divisors, and the monodromy group.%

A class of varieties for which the first two are well understood is given by Nikulin-type orbifolds.
These orbifolds are deformation equivalent to special four-dimensional orbifolds known as Nikulin orbifolds. The latter are obtained as partial resolutions in codimension 2 of quotients by symplectic involutions of $\textnormal{K3}^{[2]}$-type manifolds (see \Cref{nota}). 
The Beauville--Bogomolov form for Nikulin-type orbifolds has been calculated in \cite{Lol3} and the wall divisors have been determined in \cite{Ulrike2}. It remains to determine the monodromy group, which is the content of our first main result (compare with \Cref{thm: monodromy}; more recently an alternative proof is given by Nanni \cite{Nanni}).%

For a lattice $L$ of signature $(3,\rk L-3)$ we denote by $O^+(L)$ the subgroup of $O(L)$ consisting of \emph{orientation-preserving isometries} (see \Cref{orient preserving in O+} for an explanation of this terminology). It is known that $\Mon^2(X) \subseteq O^+(\HH^2(X,\ZZ))$ for any IHS variety $X$ (see \Cref{reminders period}).%
\begin{thm}\label{mainintro}%
Let $X$ be an irreducible symplectic orbifold of Nikulin-type and $\Mon^2(X)$ its monodromy group. Then%
\[\Mon^2(X) = O^+(\HH^2(X,\ZZ)).\]%
\end{thm}%
To prove the theorem, we first use the results of \cite[Section 6]{Ulrike2} to prove the existence of many monodromy operators given by reflections. We then use lattice theoretical arguments to show that $\Mon^2(X)$ is generated by such reflections, and that it is maximal. \Cref{mainintro} has several immediate and useful consequences, which we record in the following remarks.
\begin{rmk}%
As a consequence of the theorem, an isometry $\HH^2(X,\Z)\rightarrow \HH^2(Y,\Z)$ is a parallel transport operator if and only if it respects the natural orientations (see \cite[Section~4]{Markman} for the definitions).%
\end{rmk}%
\begin{rmk}%
The coarse moduli space of $\Lambda$-marked orbifolds of Nikulin-type, where
\[
\Lambda \coloneqq U(2)^{\oplus 3} \oplus E_8 \oplus A_1^{\oplus 2},
\]
has two connected components.
Indeed, two $\Lambda$-marked pairs $(X,\phi)$ and $(X,g\circ\phi)$, where $g$ is a monodromy operator, are related by parallel transport in a deformation family and hence belong to the same connected component of the moduli space.
Thus the connected components are indexed by $O(\Lambda)/\Mon^2(\Lambda)$, which has cardinality two, since $O^+(\Lambda)$ has index two in $O(\Lambda)$.
\end{rmk}%
We now have all the ingredients to classify symplectic birational automorphisms of Nikulin-type orbifolds. Our approach is computer-aided, and the computations are carried out using the software OSCAR \cite{oscar-book, oscar-system}. The resulting classification is described by \Cref{prop type of symp bir} and \Cref{tab: table symplectic}. Among the regular automorphisms, we find the following cases:%
\begin{itemize}%
\item
Automorphisms that are deformation equivalent to those induced from an automorphism on a manifold of $\textnormal{K3}^{[2]}$-type. We call these automorphisms \emph{standard} (see  \Cref{defistandard}).%
\item
Certain involutions that we call \emph{exceptional} (see \Cref{defiinvolution}). A geometric example of such an involution can be obtained by exchanging the role of the two exceptional divisors on a Nikulin orbifold constructed from the Hilbert scheme of 2 points on a K3 surface. This  type of involutions is introduced in \cite[Section 8.2]{Ulrike2}.%
\item
Two new types of automorphisms: One of order 5 and one involution with anti-invariant lattice $D_{10}(2)$. At the time this paper is written, no explicit geometric realizations are known for them.
\item
Compositions of the previous cases.%
\end{itemize}%

The methods developed in this paper can be adapted to classify symplectic birational automorphisms of other deformation types of IHS orbifolds. Several examples of such orbifolds have already been provided \cite{Mauri,Lol6}.\medskip

The paper is organized as follows. After providing some reminders in \Cref{reminders}, we prove \Cref{mainintro} in \Cref{monodromysection}. Finally, \Cref{classificationsection} is dedicated to establishing the classification of automorphisms of Nikulin-type orbifolds.%

\subsubsection*{Acknowledgments}%
We would like to thank the organizers of the 
workshop ``Hyperkahler quotients, singularities, and quivers'' taking place at the Simons Center for Geometry and Physics in 2023, where the idea for this paper was conceived. We also express our appreciation to the Simons Center for its generous hospitality. 
We would also like to thank the anonymous referees for their suggestions and comments, which really helped improving the quality of the paper.
The second author was supported by PRCI SMAGP (ANR-20-CE40-0026-01). The first and third authors were supported by the Deutsche Forschungsgemeinschaft (DFG, German Research Foundation) -- Project-ID 286237555 - TRR 195 (Gef\"ordert durch die Deutsche Forschungsgemeinschaft -- Projektnummer 286237555 - TRR 195).%

\section{Reminders and notation}\label{reminders}%

\subsection{Lattices}%
We define a \emph{lattice} $(L, q)$ to be a finitely generated free $\mathbb{Z}$-module $L$ equipped with a $\mathbb{Q}$-valued nondegenerate symmetric bilinear form $q$. If $q(L, L)\subseteq \mathbb{Z}$, we call the lattice \emph{integral}. A lattice $(L, q)$ is \emph{even} if $q(x,x)\in 2\mathbb{Z}$ for all $x\in L$. Even lattices are integral. If $q$ is clear from the context, we often drop $q$ from the notation, and for all $x,y\in L$, we denote $x.y \coloneqq q(x,y)$ and $q(x) \coloneqq x^2\coloneqq x.x$. We write $L(-1)$ for $(L,-q)$. Note that the orthogonal groups $O(L)$ and $O(L(-1))$ are equal. 
Given an even lattice $L$, we denote its dual by $L^\vee$, and we define its discriminant group $D_L \coloneqq L^\vee/L$. It is equipped with a torsion quadratic form%
\[q_{D_L}\colon D_L\to \mathbb{Q}/2\mathbb{Z},\, x+L\mapsto x^2+2\mathbb{Z}.\]%
The \emph{divisibility} $\mathrm{div}_L(x)$ of $x \in L$ is the natural number $n \in \mathbb{N}$ with $x.L=n\ZZ$. 
For $G \leq O(L)$ we denote by $L^G= \{x \in L | \forall g \in G: g(x) = x \}$ the \emph{invariant lattice} of $G$ and by $L_G=(L^G)^\perp$ the \emph{coinvariant lattice}. A similar notation applies to a single $g\in G$.%

We follow the convention that ADE root lattices are negative definite, and we denote by $U$ the hyperbolic plane lattice. 
We denote the orthogonal direct sum of two lattices $A,B$ by $A \oplus B$.

\subsubsection*{Subgroups of isometries}
For $(L, q)$ an integral lattice we denote by $O^\sharp(L)$ the kernel of the natural map $O(L) \to O(D_L)$, by $SO(L)$ the normal subgroup of isometries of determinant $+1$ and by $O^{\pm}(L)$ the kernel of the real spinor norm with respect to $\mp q$.

\begin{rmk}
    This last convention is typical in the literature on IHS varieties, but atypical in the quadratic forms world.
\end{rmk}

We set $O^{\sharp,\pm}(L) \coloneqq O^\sharp(L) \cap O^\pm(L)$, and $SO^{\sharp,\pm}(L) \coloneqq SO(L)\cap O^{\sharp,\pm}(L)$. We define also similarly $O^{\pm}(V)$ and $SO(V)$ for any real quadratic space $V$.

\subsubsection*{Orientation}
Let now $(L, q)$ be a lattice of signature $(3, \ast)$. We call \emph{(real) positive three-space} in $L$ any three-dimensional subvector space $W\subseteq L\otimes \mathbb{R}$ such that $q_{|W}$ is positive definite. Any choice of a basis of $W$, or an $SO(W)$-orbit of bases, determines what we call an \emph{orientation} on $W$, in which case we say that $W$ is \emph{oriented}. Let $C_L = \{x \in L \otimes \R \mid x^2 >0 \}$ be the positive cone of $L$. According to Markman \cite[Lemma 4.1]{Markman}, any positive three-space $W$ in $L$ satisfies that $W\setminus\{0\}$ is deformation retract to $C_L$, and the choice of an orientation on $W$ fixes an \emph{orientation} on $C_L$, i.e. a generator of $\HH^2(C_L, \mathbb{Z})\cong \mathbb{Z}$. The following is well known:

\begin{lemme}\label{orient preserving in O+}
    The subgroup of $O(L)$ preserving the orientation on $C_L$ coincides with $O^+(L)$.
\end{lemme}

\begin{proof}
    By the Cartan--Dieudonn\'e theorem, any isometry of $L\otimes \RR$ decomposes as the composition of finitely many reflections in anisotropic vectors of $L\otimes \mathbb{R}$. Hence it suffices to determine which reflections preserve orientations.

    Fix an orientation on $C_L$, i.e. we fix an orientation on any positive three-space in $L$. Let $v\in L\otimes \mathbb{R}$ be anisotropic and let $R_v\in O(L\otimes \mathbb{R})$ be the associated reflection. It is defined by:
    \[R_v(x) = x - 2\frac{x.v}{v^2}v,\qquad \forall x\in L\otimes\mathbb{R}.\]
    If $v^2 < 0$, meaning that $R_v\in O^+(L\otimes \mathbb{R})$, then $R_v$ acts trivially on any positive three-space in $v^\perp$. Therefore $R_v$ preserves the orientation on $C_L$. Otherwise, if $v^2 > 0$, we complete $v$ to an orthogonal basis $\{v,a,b,x_1,\ldots x_{r-3}\}$ of $L\otimes \mathbb{R}$, where $r$ is the rank of $L$ and $a,b$ are vectors of positive square. In particular, $R_v$ preserves the positive three-space $W$ spanned by $\{v,a,b\}$ and its restriction to $W$ has determinant $-1$. Hence $R_v$ does not preserve the orientation of $W$, and $R_v$ does not preserve the orientation on $C_L$. It follows that an isometry $f\in O(L)$ preserves the orientation on $C_L$ if and only if $f_{\mathbb{R}}$ decomposes into a composition of reflections where only an even number of them are defined by positive vectors. By our convention, the latter is equivalent to saying that $f\in O^+(L)$.
\end{proof}

\subsection{Orbifolds of Nilukin-type}\label{nota}%
\begin{defi}\label{Niku}%
 \item%
 Let $X$ be a manifold of $\textnormal{K3}^{[2]}$-type endowed with a symplectic involution $i$. We denote by \[\pi\colon X\rightarrow X/i \eqqcolon M\]%
 the quotient map. By \cite[Theorem 4.1]{Mongardi}, we know that $\Fix i = \Sigma \cup \left\{\ 28\ points\ \right\}$ where $\Sigma$ is a \textnormal{K3} surface. For simplicity, $\pi(\Sigma)$ is also denoted by $\Sigma$. Let $r:M'\rightarrow M$ be the blow-up of $M$ in $\Sigma$. The exceptional divisor is denoted by $\Sigma'$. The orbifold $M'$ is called a \emph{Nikulin orbifold}.%
\end{defi}%
We summarize our notation in the following diagram:%

\[
\begin{tikzcd}
    \Sigma'\arrow[d, hook]\arrow[r]&\Sigma\arrow[d, hook]&\\
    M'\arrow[r, "r"']&M&X\arrow[l, "\pi"]
\end{tikzcd}
\]

\begin{nota}%
If $X=S^{[2]}$ for $S$ a \textnormal{K3} surface, we set $\delta'\coloneqq r^*(\pi_*(\delta))$ where $\delta$ is half of the class of the diagonal in $S^{[2]}$.%
\end{nota}%
Recall that a compact K\"ahler orbifold is called irreducible symplectic if its singular locus has codimension at least $4$, and its smooth locus is simply connected and carries a unique (up to scalar) nondegenerate holomorphic 2-form.%
\begin{rmk}%
 A Nikulin orbifold is an irreducible symplectic orbifold \cite[Proposition 3.8]{Lol4}.%
\end{rmk}%
\begin{defi}%
 An irreducible symplectic orbifold which is deformation equivalent to a Nikulin orbifold is called an \emph{orbifold of Nilukin-type}.%
\end{defi}%
\subsection{Beauville--Bogomolov form}%
Let $Y$ be an irreducible symplectic orbifold. 
We recall from \cite[Theorem 3.17]{Lol4} that $\HH^2(Y,\Z)$ is endowed with a primitive nondegenerate integral bilinear form: The Beauville--Bogomolov form. This form will be denoted by $q_Y$.%
%
%
%

%
Let $M$ and $M'$ be as in \Cref{Niku}.
We equip $M$ with a Beauville--Bogomolov form by setting:
\[
q_M(x) \coloneqq q_{M'}\bigl(r^*(x)\bigr),
\qquad \forall x \in \HH^2(M,\mathbb{Z}).
\]

%
We denote by $j:\HH^2(S,\Z)\rightarrow \HH^2(S^{[2]},\Z)$ the \emph{Beauville map} (see \cite[Part 2, Proposition 6]{Beauville} for the construction). The map $j$ respects the pairings and provides an orthogonal decomposition 
$$\HH^2(S^{[2]},\Z)=j(\HH^2(S,\Z))\oplus\Z \delta.$$
We recall the following theorem. 

\begin{thm}[{{\cite[Theorem 3.6]{Ulrike2}}}]\label{BBform}%
Let $M$ and $M'$ be as in \Cref{Niku} for $X = S^{[2]}$.%
\begin{enumerate}[(i)]%
\item%
The Beauville--Bogomolov form of $M'$ is given by 
$(\HH^2(M',\Z),q_{M'})\simeq U(2)^{\oplus3}\oplus E_8\oplus A_1^{\oplus2}$ where the Fujiki constant is equal to 6.%
\item
$q_{M}(\pi_*(x))=2q_{S^{[2]}}(x)$ for all $x\in \HH^2(S^{[2]},\Z)^{\iota^{[2]}}$.%
\item
$q_{M'}(\delta')=q_{M'}(\Sigma')=-4$.%
\item
$q_{M'}(r^*(x),\Sigma')=0$ for all $x\in \HH^{2}(M,\Z)$.%
\item
$\HH^2(M',\Z)=r^*\pi_*(j(\HH^2(S,\Z)))\oplus \Z\frac{\delta'+\Sigma'}{2}\oplus \Z\frac{\delta'-\Sigma'}{2}$.%
\end{enumerate}%
\end{thm}%
\begin{rmk}%
 In the previous theorem the Beauville--Bogomolov form of $M'$ is obtained as follows:%
 \begin{itemize}%
\item%
$r^*\pi_*(j(\HH^2(S,\Z)))\simeq U(2)^{\oplus 3}\oplus E_8$.%
\item%
$\Z\frac{\delta'+\Sigma'}{2}\oplus\Z\frac{\delta'-\Sigma'}{2}%
\simeq A_1^{\oplus 2}.$%
\end{itemize}%
\end{rmk}%
\begin{nota}%
 We set $\Lambda\coloneqq U(2)^{\oplus3}\oplus E_8\oplus A_1^{\oplus2}$.%
\end{nota}%

\subsection{Deformation, period map and birational automorphisms}\label{reminders period}

\subsubsection*{Deformation and parallel transport}
\begin{defi}\label{defi}
A \emph{deformation} of an irreducible symplectic orbifold $X$ (resp.\ of a pair $(X,g)$,
where $g$ is an automorphism of $X$) is a proper and flat morphism
$s\colon \mathcal{X}\to B$, over an analytic base $B$, together with a distinguished point $o\in B$ and an
identification $X\simeq \mathcal{X}_o$, such that every fiber of $s$ is an irreducible
symplectic orbifold (resp.\ together with an automorphism $\mathfrak g$ of
$\mathcal X$ satisfying $s\circ\mathfrak g=s$ and whose restriction to the central
fiber coincides with $g$).
%
\end{defi}
\begin{rmk}
Note that a deformation of irreducible symplectic orbifolds is always locally trivial (see \cite[Proposition 3.10]{Lol4}).
\end{rmk}


\begin{defi}\label{transp}
Let $X_1$ and $X_2$ be two irreducible symplectic orbifolds. We call an isometry $f\colon\HH^{2}(X_{1},\Z)\rightarrow \HH^{2}(X_{2},\Z)$ \emph{parallel transport operator} if there exists a deformation $s\colon\mathcal{X}\rightarrow B$, two points $b_1,\,b_2\in B$, two isomorphisms $\psi_{i}\colon X_{i}\rightarrow \mathcal{X}_{b_{i}}$, $i=1,2$ and a continuous path $\gamma\colon\left[0,1\right]\rightarrow B$ with $\gamma(0)=b_{1}$, $\gamma(1)=b_{2}$ and such that the parallel transport in the local system $R^2s_{*}\Z$ along $\gamma$ induces the morphism $\psi_{2*}\circ f\circ\psi_{1}^{*}\colon \HH^{2}(\mathcal{X}_{b_{1}},\Z)\rightarrow \HH^{2}(\mathcal{X}_{b_{2}},\Z)$.
\end{defi}
\begin{defi}\label{Mon}
Let $X$ be an irreducible symplectic orbifold. The group of parallel transport operators from $X$ to itself is a subgroup of $O(\HH^2(X,\Z))$, called the \emph{monodromy group} of $X$, and denoted by $\Mon^2(X)$. 
\end{defi}
\begin{rmk}
We recall that monodromy operators preserve the Beauville--Bogomolov form, since this
form is uniquely characterized by the Fujiki relation, which is invariant under
parallel transport in deformation families.
\end{rmk}
\begin{rmk}\label{orientation}
%
%
%
%
We sketch a proof of the well-known fact that $\Mon^2(X)\subseteq O^+(\HH^2(X,\mathbb Z))$.
Let $f\in\Mon^2(X)$. Then there exist a deformation $s\colon\mathcal X\to B$ of $X$ and a loop
$\gamma\colon [0,1]\to B$ based at $o\in B$ such that the parallel transport along $\gamma$
induces $f$ on $\HH^2(X,\mathbb Z)$. 
Pulling back the local system $R^2 s_*\mathbb R$ via the path
$\gamma\colon [0,1]\to S$ yields a local system on $[0,1]$, which is necessarily trivial
since $[0,1]$ is simply-connected. Hence the cohomology groups
$\HH^2(\mathcal X_{\gamma(t)},\mathbb R)$ can be canonically identified with $\HH^2(X,\mathbb R)$. 

Fix $0\neq\sigma_o\in \HH^{2,0}(X)$ and a K\"ahler class $\omega_o$, and set
$P_o\coloneqq\langle \Rea(\sigma_o),\Ima(\sigma_o),\omega_o\rangle\subset \HH^2(X,\mathbb R)$.
This is a positive three-space, oriented by the basis $(\Rea(\sigma_o),\Ima(\sigma_o),\omega_o)$.
Along $\gamma(t)$, the real plane $\Pi_t\coloneqq\langle \Rea(\sigma_{\gamma(t)}),\Ima(\sigma_{\gamma(t)})\rangle$
varies continuously, and one can choose locally a continuous family of K\"ahler classes
$\omega_{\gamma(t)}$, so the oriented positive three-space
$$P_t\coloneqq\langle \Rea(\sigma_{\gamma(t)}),\Ima(\sigma_{\gamma(t)}),\omega_{\gamma(t)}\rangle$$
varies continuously in the Grassmannian of positive three-spaces.
Since the two orientations correspond to two connected components, the orientation cannot
change along a continuous path; hence the parallel transport $f$ preserves the orientation,
i.e.\ $f\in O^+(H^2(X,\mathbb Z))$ (see \Cref{orient preserving in O+}).
%
\end{rmk}


\subsubsection*{Moduli space of marked irreducible symplectic orbifolds}
Let $X$ be an irreducible symplectic orbifold of a given abstract deformation type denoted $T$, and let $\Lambda_T$ be a lattice of signature $(3, b_2(X)-3)$ such that $\HH^2(X, \Z)\simeq \Lambda_T$. Any fixed isometry $\eta\colon \HH^2(X, \Z)\to \Lambda_T$ is called a \emph{marking}, and we call the pair $(X, \eta)$ a $\Lambda_T$-marked pair. 
We say that two $\Lambda_T$-marked pairs $(X, \eta)$ and $(X', \eta')$ are \emph{equivalent} if there exists an isomorphism $f\colon X'\to X$ such that $\eta = \eta'\circ f^*$.%

There is a coarse moduli space $\mathcal{M}_{\Lambda_T}$ parametrizing equivalence classes of $\Lambda_T$-marked pairs of irreducible symplectic orbifolds. It is equipped with the \emph{period map}%
\[\mathscr{P}_{\Lambda_T}\colon \mathcal{M}_{\Lambda_T}\to \mathcal{D}_{\Lambda_T},\ [(X, \eta)]\mapsto \eta(\HH^{2,0}(X))\]%
where%
\[\mathcal{D}_{\Lambda_T} \coloneqq \left\{\mathbb{C}\omega\in \mathbb{P}(\Lambda_T\otimes\mathbb{C})\;\mid\; (\omega^2=0)\wedge (\omega.\overline{\omega}>0)\right\}\]%
is the so-called \emph{period domain}.

\begin{prop}[{{{\cite[Theorem 5.9]{Lol4}}}}]\label{surjectivity period map}%
    The period map $\mathscr{P}_{\Lambda_T}$ is surjective when restricted to any connected component of $\mathcal{M}_{\Lambda_T}$.
\end{prop}%

The moduli space $\mathcal{M}_{\Lambda_T}$ is neither connected nor Hausdorff. Two points $[(X, \eta)]$ and $[(X', \eta')]$ lie in the same connected component if and only if $\eta^{-1}\circ\eta'\colon \HH^2(X', \Z)\to \HH^2(X, \Z)$ is a parallel transport operator. If one of the latter holds, then $[(X, \eta)]$ and $[(X', \eta')]$ are non-separated if and only if $\mathscr{P}_{\Lambda_T}([(X, \eta)]) = \mathscr{P}_{\Lambda_T}([(X', \eta')])$.
An example of two non-separated orbifolds of Nikulin-type is provided in \cite[Section 3.3]{Ulrike0}.%
\begin{prop}[{{{\cite[Proposition 3.22]{Lol4}}}}]%
    Two non-separated distinct points in a given connected component of $\mathcal{M}_{\Lambda_T}$ determine birational irreducible symplectic orbifolds.%
\end{prop}%

\subsubsection*{Twistor space}
For any positive three-space $W\subset \Lambda_T\otimes\R$, we define the associated \emph{twistor line} $T_W\subset \mathcal{D}_{\Lambda_T}$ by:
$$T_W\coloneqq\mathcal{D}_{\Lambda_T}\cap \mathbb{P}(W\otimes\C).$$
A twistor line is called \emph{generic} if $W^{\bot}\cap \Lambda_T=0$. 
\begin{thm}[\cite{Lol4}, Theorem 5.4]\label{Twistor}
  Let $(X,\eta)$ be a $\Lambda_T$-marked irreducible symplectic orbifold. Let $\alpha$ be a K\"ahler class on $X$, and $W_\alpha\coloneqq\Vect_{\R}(\eta(\alpha),$ $\eta(\Rea \sigma_X),\eta(\Ima \sigma_X))$. 
Then:
\begin{enumerate}[(i)]
\item
There exists a metric $g$ and three complex structures (see \cite[Section 5.1]{Lol4} for the definition) $I$, $J$ and $K$ in quaternionic relation on $X$ such that:
$$\alpha= \left[g(\cdot,I\cdot)\right]\ \text{and}\ g(\cdot,J\cdot)+ig(\cdot,K\cdot)\in \HH^{0,2}(X).$$
\item
There exists a deformation of $X$: 
$$\mathscr{X}\rightarrow T(\alpha)\simeq\mathbb{P}^1,$$ such that the period map
$\mathscr{P}\colon T(\alpha)\rightarrow T_{W_\alpha}$ provides an isomorphism. Moreover, for each $s=(a,b,c)\in \mathbb{P}^1$, the associated fiber $\mathscr{X}_s$ is an orbifold diffeomorphic to $X$ endowed with the complex structure $aI+bJ+cK$.
\end{enumerate}
\end{thm}
\begin{defi}\label{Twistordef}
 Such a deformation $\mathscr{X}\rightarrow T(\alpha)\simeq\mathbb{P}^1$ as in \Cref{Twistor} is called a \emph{twistor space}. When $W_\alpha$ is generic, we say that $\mathscr{X}\rightarrow T(\alpha)\simeq\mathbb{P}^1$ is a \emph{generic twistor space}.
\end{defi}

\subsubsection*{Automorphisms and birational automorphisms}

Let $X$ be an irreducible symplectic orbifold. We denote by $\Aut(X)\leq \Bir(X)$ the groups of automorphisms, respectively, birational automorphisms of $X$.%

\begin{rmk}%
    By abuse of terminology, even though $X$ is not necessarily projective, we talk about isomorphisms and birational isomorphisms instead of biholomorphic and bimeromorphic maps.%
\end{rmk}%

The image of the group homomorphism%
\[\rho_X\colon \Bir(X)\to \GL(\HH^2(X, \Z)),\, f\mapsto (f^*)^{-1}\]%
is contained in the subgroup of $\Mon^2(X)$ consisting of monodromies preserving the Hodge structure on $\HH^2(X, \C)$. It is known to be injective for $X$ of Nikulin-type \cite[Proposition 8.1]{Ulrike2}.%

\begin{nota}%
    We denote by $\Bir_s(X)$ the subgroup of birational automorphisms of $X$ whose action on $\HH^{2,0}(X)$ is trivial: We call such birational automorphisms \emph{symplectic}. Any other birational automorphism is said to be \emph{nonsymplectic}. We denote moreover by $\Aut_s(X)\leq \Bir_s(X)$ the subgroup consisting of symplectic automorphisms of $X$.%
\end{nota}%

\section{The monodromy group}\label{monodromysection}%
In this section, we show that the monodromy group of Nikulin-type orbifolds is maximal.%
\subsection{Reflections as monodromy operators}%
Let $X$ be an irreducible symplectic orbifold of Nikulin-type. Let $\Mon^2(X)\leq O^+(\HH^2(X, \mathbb{Z}))$ denote the monodromy group of $X$. 
Fix a marking $\eta\colon \HH^2(X, \mathbb{Z})\to \Lambda=U(2)^{\oplus3}\oplus E_8\oplus A_1^{\oplus2}$, and let%
\[\Mon^2(\Lambda, \eta) \coloneqq \eta\Mon^2(X)\eta^{-1}.\]%
Note that a priori $\Mon^2(\Lambda, \eta)\leq O^+(\Lambda)$ is not a normal subgroup. We show later in \Cref{thm: monodromy} that it actually is, and that the definition of $\Mon^2(\Lambda, \eta)$ does not depend on $\eta$.%

The following is a corollary of \cite[Theorem 6.15]{Ulrike2}: It is essential in order to prove \Cref{mainintro}. 

\begin{cor}\label{cor:reflections are monodromies}%
 Let $v\in \Lambda$ which verifies one of the two following properties:%

 \begin{enumerate}[(i)]%
  \item $v^2=-2$;
  \item  $v^2=-4$ and $\Div_\Lambda(v)=2$.%
  \end{enumerate}%
Let $R_v$ be the reflection $$R_v\colon\Lambda\rightarrow \Lambda,\ x\mapsto x-\frac{2v.x}{v^2}v.$$ Then $R_v\in \Mon^2(\Lambda, \eta)$.%
\end{cor}%
\begin{proof}%
First we recall the notation introduced in \cite{Ulrike2} and used in this proof.
Let $i\in \mathbb{N}$. 
We set $L^{(2)}_i$ a primitive element in $U(2)^{\oplus 3}\subset \Lambda$ of square $4i$. We also set $e_1^{(1)}$ and $e_2^{(1)}$ primitive elements in $E_8\subset \Lambda$ respectively of square $-2$ and $-4$. 
Let now $M'$ be a Nikulin orbifold as in \Cref{Niku}. We fix a marking $\eta$ and identify $\HH^2(M',\ZZ)$ with $\Lambda$ via $\eta$. 
We denote by $\left(\frac{\delta'+\Sigma'}{2},\frac{\delta'-\Sigma'}{2}\right)$ a basis of $A_1^{\oplus 2}\subset \Lambda$. 

Let $T$ be the set of elements in $\Lambda$ of square $-2$ or of square $-4$ and divisibility 2.  
According to \cite[Remark 6.14, Theorem 6.15]{Ulrike2}, the action of $\Mon^2(\Lambda, \eta)$ on $\Lambda$ decomposes $T$ into at most 6 orbits, whose representatives are given in the following table.%

\begin{center}%
{
\setlength{\tabcolsep}{7pt}%
\renewcommand%
\arraystretch{1.5}%
\rowcolors{1}{white}{lightgray!40!white}%
 \begin{tabular}{cccc|cccc}%
\hline%
     &$v$&$v^2$&$\text{div}_\Lambda(v)$&&$v$&$v^2$&$\text{div}_\Lambda(v)$\\%
    \hline%

    (1)&$L_{-1}^{(2)}$&$-4$&2&(4)&$\frac{\delta'+\Sigma'}{2}$&$-2$&2\\%

    (2)&$\delta'$&$-4$&2&(5)&$L_1^{(2)}+e_2^{(1)}-\frac{\delta'+\Sigma'}{2}$&$-2$&1\\%

    (3)&$2L_1^{(2)}+2e_2^{(1)}-\delta'$&$-4$&2&(6)&$e_1^{(1)}$&$-2$&1\\%

\hline%
\end{tabular}%
}
\end{center}%

Assume that $y=f(x)$ with $R_x\in \Mon^2(\Lambda, \eta)$ and $f\in \Mon^2(\Lambda, \eta)$. Then $R_y\in \Mon^2(\Lambda, \eta)$. Indeed, $R_y=f\circ R_x \circ f^{-1}$.%

Therefore, we only need to prove that the reflections through each element listed above is a monodromy operator. 
We can see each element in $\Lambda$ as an element in $\HH^2(M',\Z)$ with $M'$ obtained via $S^{[2]}$ for a K3 surface $S$. 
According to \Cref{BBform}, with suitable $L_{-1}, e_1 \in H^2(S^{[2]},\ZZ)$ we can write:%
\begin{itemize}%
 \item[(1)]%
 $L_{-1}^{(2)}=r^*(\pi_*(L_{-1}))$ with $q_{S^{[2]}}(L_{-1})=-2$,%
 \item[(2)] $\delta'=r^*(\pi_*(\delta))$ with $q_{S^{[2]}}(\delta)=-2$,%
 \item[(6)]%
 $e_1^{(1)}=r^*(\pi_*(e_1))$ with $q_{S^{[2]}}(e_1)=-2$.%
\end{itemize}%
By \cite[Theorem 9.1]{Markman}, we know that a reflection through an element of square $-2$ is a monodromy operator on $\HH^2(S^{[2]},\Z)$. By \cite[Proposition~3.9]{Ulrike2}, monodromy operators of $S^{[2]}$ commuting with the
Nikulin involution descend to monodromy operators on $M'$, and hence the reflections
through the elements in cases (1), (2) and (6) are monodromy operators.
%
According to \cite[Remark 4.7]{Ulrike2}, we also know in case (4) that $R_{\frac{\delta'+\Sigma'}{2}}\in \Mon^2(\Lambda, \eta)$.%

It remains to treat the cases (3) and (5). 
Let $A'\coloneqq L_1^{(2)}+e_2^{(1)}-\frac{\delta'+\Sigma'}{2}$ and $B'\coloneqq2L_1^{(2)}+2e_2^{(1)}-\delta'$. By \cite[Lemma 5.25]{Ulrike2}, we know that $R_{A'}\in\Mon^2(\Lambda, \eta)$ and
\begin{align*}
R_{A'}(\Sigma')&=\Sigma'-\frac{2\Sigma'. A'}{(A')^2}A'\\
&=\Sigma'-\frac{2\times 2}{-2}\left(L_1^{(2)}+e_2^{(1)}-\frac{\delta'+\Sigma'}{2}\right)\\
&=B'.
\end{align*}
However, we know that $R_{\Sigma'}\in \Mon^2(\Lambda, \eta)$ according to \cite[Corollary 4.6]{Ulrike2}. Hence $R_{B'}=R_{A'}\circ R_{\Sigma'}\circ R_{A'}\in \Mon^2(\Lambda, \eta)$.
This concludes the proof.%

\end{proof}%

\subsection{The orthogonal group of the discriminant form}%
Recall that we have fixed $\Lambda \coloneqq U(2)^{\oplus3} \oplus E_8 \oplus A_1^{\oplus 2}$. 
Its discriminant group $D_\Lambda$ is an $\FF_2$-vector space of dimension $8$ and comes equipped with the discriminant form.%

Let $b$ be a nondegenerate symplectic form on a finite dimensional $\FF_2$-vector space $V$. 
For $0\neq u \in V$ we call $T_u \in \mathrm{Sp}(b)$ defined by%
\[T_u(x) = x+b(u,x)u\]%
a \emph{symplectic transvection}. 
It is well known (see e.g. \cite[Theorem 2.1.9]{omeara}) that symplectic transvections generate $\mathrm{Sp}(b)$. 
For $q\colon V \to \Q/2\mathbb{Z}$ a torsion quadratic form, its \emph{polar form} is%
\[b\colon V \times V \to \Q/2\ZZ, \quad b(x,y)\coloneqq q(x+y)-q(x)-q(y),\]%
which, by the definition of a torsion quadratic form, is required to be bilinear. 
For $V$ an $\FF_2$-vector space the polar form actually takes values in $\ZZ/2\ZZ \cong \FF_2$.  It is alternating if and only if  $q(V)=\ZZ/2\ZZ$. For $W \leq V$ a submodule, we denote by $W^\perp$ its orthogonal with respect to the polar form $b$.%

For $u\in V$, a short calculation shows that the symplectic transvection $T_u$ preserves $q$, i.e., $\forall x \in V\colon q(x)=q(T_u(x))$, if and only if $q(u)=1+2\Z$. In this case $T_u$ is called a \emph{reflection}.%

The torsion quadratic form $q|_{D_\Lambda}$ takes values in $\frac{1}{2} \ZZ /2\ZZ$. Hence its polar $b$ is not alternating. 
\begin{lemme}\label{lem:OD_Lambda_generated_by_reflections}%
 The group $O(D_\Lambda)$ is generated by reflections.%
\end{lemme}%
\begin{proof}%
For this proof, we let $x,y$ be generators of the two copies of $A_1$ in $\Lambda$. Consider the codimension one subspace%
\[K\coloneqq\{x \in D_\Lambda \mid 2q(x) = 0 + 2\Z\}\leq D_\Lambda.\]%
The space $K$ is the largest subspace where the bilinear form $b|_{K\times K}$ is alternating. It has a one dimensional radical $R = K \cap K^\perp$ generated by the single element $r\coloneqq \frac{x+y}{2}+\Lambda$, which satisfies $q(r)=1 +2\ZZ$. 
 The polar form $b$ thus descends to a nondegenerate alternating form on $K/R$, i.e. to a symplectic form.%

 Let $u\in K\setminus R$ and $T_{u+R} \in \mathrm{Sp}(K/R)$ be the corresponding symplectic transvection; note that $u+R =\{u,u+r\}$. Since $b(u,r)=0+2\ZZ$ and $q(r)=1+2\Z$, we have $q(u+r)=q(u)+1+2\ZZ$. Therefore, we find a unique $v \in u+R$ with $q(v)=1+2\Z$. 
 The reflection $T_v$ induces the symplectic transvection $T_{u+R} \in \mathrm{Sp}(K/R)$. 
 Let $S$ denote the subgroup of $O(D_\Lambda)$ generated by reflections. 
 Since $\mathrm{Sp}(K/R)$ is generated by symplectic transvections, the previous considerations show that the natural map $S \to \mathrm{Sp}(K/R)$ is surjective.%

 By \cite[Proposition 2.6]{brandhorst-veniani} and its proof we know that%
 \begin{equation}\label{eqn:ODtoSp}
     O(D_\Lambda) \to \mathrm{Sp}(K/R)
 \end{equation}%
 is surjective and its kernel has order $2$. 
  Hence, $[O(D_\Lambda):S] \in \{1,2\}$. 
 The element $z= x+y\in A_1^{\oplus2}$ has $z^2 =-4$ and divisibility $2$, and it satisfies $r = z/2+\Lambda$. The reflection $R_z$ in $z$ induces the reflection $T_r\in S$, which lies in the kernel of $O(D_\Lambda) \to \mathrm{Sp}(K/R)$ because $T_r(a)=a+b(r,a)r\equiv a \mod R$ for all $a\in K$.
 This proves that $O(D_\Lambda)=S$.%
\end{proof}%

\subsection{Generation of the monodromy group by reflections}%
Recall that $\Lambda \coloneqq U(2)^{\oplus3} \oplus E_8 \oplus A_1^{\oplus 2}$.
\begin{lemme}\label{lem:exists_lift_square_4}%
 Let $u \in D_\Lambda$ with $q(u)=1$. Then there exists $x \in \Lambda$ with $x^2=-4$, $\mathrm{div}_\Lambda(x)=2$ and $u = x/2+\Lambda$.%
\end{lemme}%
\begin{proof}%
 It is well known (see e.g. \cite[2.1.3]{omeara}) that the symplectic group $\mathrm{Sp}(K/R)$ has two orbits on 
 $K/R$, one represented by $0+R$ and the other one by $u+R$ with $u \not\in R$. Recall from \cref{eqn:ODtoSp} that $O(D_\Lambda) \to \mathrm{Sp}(K/R)$ is surjective. 
 Hence for $u,v \in K\setminus R$ with $q(u)=q(v)=1+2\ZZ$ we find $f \in O(D_\Lambda)$ with $u-f(v) \in R=\{0,r\}$. Suppose that $u-f(v)=r$. With the fact that $r \in K^\perp$, we arrive at the contradiction%
 \[1+2\ZZ=q(u)=q(f(v)+r)=q(f(v))+q(r)=q(v)+q(r) = 0+2\ZZ.\]%
 Thus $u=f(v)$. 
 This shows that the action of $O(D_\Lambda)$ has two orbits on the set%
 \[\Gamma = \{v \in D_\Lambda \mid q(v)=1+2\ZZ\}.\]%
 The first orbit consists of one element, the generator $r$ of $R$. 
 The second orbit is represented by any element of $\Gamma \setminus \{r\}$.%

 Let $w\in \Lambda$ of square $-4$ be in one of the copies of $U(2)$. 
 Let $x,y$ generate the two copies of $A_1$ in the definition of $\Lambda$. Then $z=x+y$ is of square $-4$ and divisibility $2$. Further $r = z/2+\Lambda$. 
 By \cite[Theorem 1.14.2]{nikulin} we know that $O(\Lambda) \to O(D_\Lambda)$ is surjective. Therefore%
 \[(w/2) O(\Lambda) \cup (z/2) O(\Lambda)\]%
 surjects onto $\Gamma$.%
\end{proof}%

For the proof of the next lemma we need to introduce some terminology on spinor norms.
Let $(L,q)$ be an integral lattice and $V \coloneqq L\otimes \mathbb{Q}$ be the corresponding quadratic space. The \emph{rational spinor norm} 
\[\theta \colon O(V) \to \QQ^\times /(\mathbb{Q}^\times)^2\] 
is the group homomorphism which for any anisotropic $v\in V$ has $\theta(R_v)=q(v,v)/2$, where $R_v$ is the reflection defined by $v$. Since $O(V)$ is generated by reflections, this uniquely characterizes $\theta$. 
Note that the reflection in a $(2)$-vector $v$ has spinor norm $1=q(v,v)/2$. 

Its $p$-adic variant is denoted by
\[\theta_p \colon O(V\otimes \QQ_p) \to \QQ_p^\times /(\mathbb{Q}_p^\times)^2.\]
We denote by $\Theta(V) = \ker \theta$ and $S\Theta(V) = \Theta(V) \cap SO(V)$.
Set $\Theta(L) = O(L) \cap \Theta(V)$ and $S\Theta(L)=S\Theta(V) \cap O(L)$. 
In the notation of O'Meara \cite[\S43]{om73} and Kneser \cite{Kneser:generation_by_reflection}, whose results we use for the next lemma, 
we have $O'(L) = S\Theta(L)$ (they define $O^+(L)$ as the special orthogonal group $SO(L)$, we do not). 

Recall that the Witt index of a quadratic space is the dimension of a maximal totally isotropic subspace.

Since the reflection in a $(-2)$-root acts trivially on the discriminant group,
the following lemma has a chance to be true.
\begin{lemme}\label{lem:stable_gen_by_reflections}
The group $O^{\sharp,+}(\Lambda)$ is generated by reflections in $(-2)$-roots. 
\end{lemme}

\begin{proof}
We work with the lattice $L\coloneqq \Lambda(-1)$ and look at reflections in $(2)$-roots.
In temporary notation let $R(L) \subseteq O(L)$ denote the subgroup they generate. Then $SR(L)=R(L)\cap SO(L)$ is the subgroup of elements that can be written as a product of an even number of reflections. 

The lattice $L$ is even, $L \otimes \R$ has Witt index $3\geq 2$ and $L$ contains a sublattice isomorphic to $E_8(-1)$. In particular, $\det(E_8)=1$ is divisible by neither $2$ or $3$ and its rank is $8 \geq 5,6$. Hence, Kneser's result \cite[Satz 4]{Kneser:generation_by_reflection} applies and gives that $S\Theta^\sharp(L)  = SR(L)$. Since the reflection in a $(2)$-vector has determinant $-1$, it generates $\Theta^\sharp(L)/ S\Theta^\sharp(L)$ and we have \[\Theta^\sharp(L)=R(L).\]
It remains to show that 
\[O^{\sharp,-}(L) = \Theta^\sharp(L).\]
The inclusion $\supseteq$ holds by the definitions. For the other inclusion $\subseteq$ we have to see that $O^{\sharp,-}(L)$ lies in the kernel of the rational spinor norm $\theta$.

By \cite[Satz 6]{Kneser:generation_by_reflection}, for all primes $p$ we have an inclusion: $\theta_p(O^{\sharp}(L_p)) \subseteq \ZZ_p^\times/(\ZZ_p^{\times})^2$. As a consequence, any isometry 
$f \in O^\sharp(L)$ satisfies $\theta(f) \in \ZZ_p^\times/(\ZZ_p^{\times})^2$ for all primes $p$. In particular $\theta(f)$ is represented by an integer which is not divisible by any prime number, i.e. $\theta(f)=\pm1  (\mathbb{Q}^\times)^2$. If $f$ is moreover in $O^{\sharp,-}(L)$, then the real spinor norm of $f$ is positive and therefore $\theta(f)=-1(\mathbb{Q}^\times)^2$ is impossible.
 \end{proof}
\begin{prop}\label{prop:reflection subgroup}%
 Set%
\[\Delta = \{x \in \Lambda \mid x^2 =-2 \vee (x^2=-4 \wedge \mathrm{div}_\Lambda(x) =2)\}.\]%
Denote by $W(\Delta)$ be the subgroup of $O(\Lambda)$ generated by reflections in $\Delta$. Then $O^+(\Lambda) = W(\Delta)$.%
\end{prop}%
\begin{proof}
 By \Cref{lem:stable_gen_by_reflections} we have%
 \begin{equation}\label{eqn:W}
  O^{\sharp,+}(\Lambda) \subseteq W(\Delta) \subseteq O^+(\Lambda).
 \end{equation} 
 Let $x \in \Lambda$ with $x^2=-4$ and $\mathrm{div}_\Lambda(x)=2$. This means that $y=x/2 \in \Lambda^\vee$ and has $y^2=-1$. The reflection $R_x \in W(\Delta)$ induces the reflection in $T_{y+\Lambda} \in O(D_\Lambda)$. 
 By \Cref{lem:OD_Lambda_generated_by_reflections} and \Cref{lem:exists_lift_square_4} this implies that $W(\Delta)$ surjects onto $O(D_\Lambda)$ and so does the potentially larger group $O^{+}(\Lambda)$. The kernel of 
 $O^+(\Lambda) \to O(D_\Lambda)$ is by definition $O^{\sharp,+}(\Lambda)$.
 With \cref{eqn:W} and the homomorphism theorem we obtain that
 \[W(\Delta)/O^{\sharp,+}(\Lambda) \cong O(D_\Lambda) \cong O^+(\Lambda)/O^{\sharp,+}(\Lambda).\] 
 This proves the claim $W(\Delta)=O^+(\Lambda)$.%
\end{proof}%

\begin{thm}\label{thm: monodromy}%
 $\Mon^2(\Lambda, \eta) = O^+(\Lambda)$.%
\end{thm}%
\begin{proof}%
By \Cref{cor:reflections are monodromies} $W(\Delta) \subseteq \Mon^2(\Lambda, \eta) \subseteq O^+(\Lambda)$. 
By \Cref{prop:reflection subgroup} $W(\Delta)=O^+(\Lambda)$.%
\end{proof}%
It follows from \Cref{thm: monodromy}, that $\Mon^2(\Lambda,\eta)$ is independent of the marking $\eta$. 
We may therefore write $\Mon^2(\Lambda)\coloneqq\Mon^2(\Lambda,\eta)$.%

\section{Classification of symplectic birational automorphisms}\label{classificationsection}%
In the previous section, we have determined the monodromy group for Nikulin-type orbifolds. From this, and thanks to the global Torelli theorem for irreducible symplectic orbifolds, we can now classify symplectic birational automorphisms of finite order for such orbifolds.%

\subsection{Torelli setting}%
For an orbifold $X$ of Nikulin-type, we recall that the map%
\[\rho_X\colon\textnormal{Bir}(X)\to O(\HH^2(X, \mathbb{Z})),\; f\mapsto (f^*)^{-1}\]%
is injective by \cite[Proposition 8.1]{Ulrike2}. Its image is described by the global Torelli theorem for irreducible symplectic orbifolds \cite[Theorem 1.1]{Lol4} in terms of the Hodge structure, the wall divisors and the monodromy group, which we calculated in \Cref{thm: monodromy}.%

We call \emph{effective} any isometry, or group of isometries, of $\Lambda$ contained in $\eta\text{im}(\rho_X)\eta^{-1}$ where $(X, \eta)$ is any $\Lambda$-marked pair. 
We call an effective isometry \emph{symplectic} (resp. \emph{regular symplectic}) if it arises by representing symplectic birational automorphisms (resp. symplectic regular automorphisms) on $\Lambda$. One can classify birational automorphisms on Nikulin-type orbifolds by studying effective isometries of the lattice $\Lambda$. Similarly to \cite[Theorem 1.2]{Ulrike2}, we define the respective sets of numerical \emph{prime exceptional divisors} and numerical \emph{wall divisors} for orbifolds of Nikulin-type to be respectively%
\begin{align*}%
    &\mathcal{W}(\Lambda)^{pex} \coloneqq \left\{x \in \Lambda \mid (x^2 =-2\wedge \mathrm{div}_\Lambda(x)=1) \vee (x^2=-4 \wedge \mathrm{div}_\Lambda(x)=2)\right\}\text{ and}\\%
    &\mathcal{W}(\Lambda) \coloneqq \mathcal{W}(\Lambda)^{pex}\cup \left\{x\in\Lambda\;\mid\; x^2=-6 \wedge \mathrm{div}_\Lambda(x)=2\right\}\\%
    &\phantom{\mathcal{W}(\Lambda) \coloneqq \mathcal{W}(\Lambda)^{pex}}\cup \left\{x\in\Lambda\;\mid\; x^2=-12 \wedge \mathrm{div}_\Lambda(x)=2\wedge x_{U(2)^{3}}\in2U(2)^{3}\right\},%
\end{align*}%
where $x_{U(2)^{3}}$ denotes the image of $x$ under the projection $\Lambda\to U(2)^3$ 
\cite[Remark 1.3]{Ulrike2}. 

\begin{rmk}
    We comment on the different types of (numerical) wall divisors just introduced.
\begin{itemize}
\item  An example of prime exceptional divisor of square $-2$ and divisibility 1 is provided in \cite[Proposition 5.22]{Ulrike2}.
\item An example of prime exceptional divisor of square $-4$ and divisibility 2 is given by $\Sigma'$ in \Cref{Niku}.
\item  An example of a contractible extremal curve dual to a class of square $-6$ and divisibility 2 is described in \cite[Lemma 5.14, Proposition 5.15]{Ulrike2}.
\item  An example of a contractible extremal curve dual to a class of square $-12$ and divisibility 2 is described in \cite[Lemma 5.21, Proposition 5.22]{Ulrike2}.
\end{itemize}
Note that a Beauville--Bogomolov class corresponding to a prime exceptional divisor remains exceptional on any deformation where it stays of type (1,1). Note also that the classes of square $-6$ or $-12$ and with divisibility 2 cannot be prime exceptional because the real reflections they define are not integral (see \cite[Theorem 3.10]{Lehn}).
\end{rmk}

The following result is standard.

\begin{lemme}\label{lemme symplectic}%
    A finite order isometry $f\in O^+(\Lambda)$ is symplectic (resp. regular symplectic) effective if and only if $\Lambda_f$ is negative definite and%
    \[\Lambda_f\cap\mathcal{W}^{pex}(\Lambda)=\emptyset \quad (\textnormal{resp. }\Lambda_f\cap\mathcal{W}(\Lambda)=\emptyset).\]%
\end{lemme}%
\begin{proof}%
    The proof is verbatim the same as in the case of IHS manifolds. For the reader's convenience, we recall the proof of the sufficient part; see for instance \cite[Theorem 6.9]{mul25} for further details.\smallskip


    Let $f\in O^+(\Lambda)$ be of finite order, and suppose that $\Lambda_f$ is negative definite and contains no elements of $\mathcal{W}^{pex}(\Lambda)$ (resp. $\mathcal{W}(\Lambda)$). The real quadratic space $\Lambda^f\otimes\mathbb{R}$ has signature $(3,\ast)$, so we can find a general enough $f$-invariant isotropic vector $\omega\in \Lambda\otimes \mathbb{C}$ such that $\omega.\overline{\omega} > 0$ and
    \[(\RR\Rea(\omega)\oplus \RR\Ima(\omega))\cap \Lambda = \Lambda^f.\]
    By the surjectivity of the period map $\mathscr{P}_\Lambda$ (\Cref{surjectivity period map}), there exists a $\Lambda$-marked Nikulin-type orbifold $(X, \eta)$ with period $\mathbb{C}\omega$: It satisfies that $\eta(\textnormal{T}(X)) = \Lambda^f$. We define $f_X \coloneqq \eta^{-1}f\eta$: It lies in $O^+(\HH^2(X, \mathbb{Z}))$ and it preserves the Hodge structure, by definition. Moreover, since $f$ fixes a positive three-space, there exists a positive $(1,1)$-class $\kappa\in \HH^{1,1}(X, \mathbb{R})\cap \eta^{-1}(\Lambda^f\otimes \mathbb{R})$ which is fixed by $f_X$. By the assumption $\Lambda_f\cap\mathcal{W}^{pex}(\Lambda)=\emptyset \quad (\textnormal{resp. }\Lambda_f\cap\mathcal{W}(\Lambda)=\emptyset)$, one of $\kappa$ or $-\kappa$ lies in an exceptional chamber (resp. in a K\"ahler-type chamber) of the positive cone $\mathcal{C}_{X}$. Hence, up to changing $X$ to a bimeromorphic model, we have that $f_X$ preserves the fundamental exceptional chamber (resp. the K\"ahler cone) of $X$. By the Hodge version of the global Torelli theorem for irreducible symplectic orbifolds \cite[Theorem 1.1]{Ulrike0}, the isometry $f_X$ is induced by a symplectic birational (resp. regular) automorphism of $X$ and $f$ is effective (resp. regular effective).
\end{proof}%

\begin{rmk}
    We have seen in the previous section that $\Mon^2(\Lambda) = O^+(\Lambda)$ is maximal in $O(\Lambda)$: In particular, any isometry of $\Lambda$ with negative definite coinvariant lattice lies in $\Mon^2(\Lambda)$ (see the proof of \Cref{orient preserving in O+}). 
Since $O(\Lambda)/O^+(\Lambda)$ is generated by $-\id_\Lambda O^+(\Lambda)$ and $-\id_\Lambda$ lies in the center of $O(\Lambda)$, the $O^+(\Lambda)$-conjugacy classes are $O(\Lambda)$-conjugacy classes. 
Therefore, a classification of symplectic isometries up to $\Mon^2(\Lambda)$-conjugation is the same as up to $O(\Lambda)$-conjugation.%
\end{rmk}

\subsection{Computations and first consequences}\label{tablesection}%
Given an even lattice $L$, the first named author and Hofmann \cite{brandhorst-hofmann} describe an algorithm to compute a complete set of representatives of $O(L)$-conjugacy classes of isometries of finite order (with at most two distinct prime divisors). The algorithms have been implemented in the OSCAR package \cite[QuadFormAndIsom]{oscar-system} by the third named author.%

\begin{thm}\label{thm classification}%
    There exist exactly 32 conjugacy classes of symplectic effective isometries $f\in \Mon^2(\Lambda)$. Among them, 21 classes are represented by regular symplectic effective isometries. The associated numerical data is available in \Cref{tab: table symplectic}. Representatives of the conjugacy classes in terms of matrices are found in the folder ``symplectic'' in the database \cite{database}.%
\end{thm}%

\begin{proof}%
    We apply the algorithm of \cite{brandhorst-hofmann} to the lattice $\Lambda \coloneqq U(2)^{\oplus 3}\oplus E_8\oplus A_1^{\oplus2}$, and compute a complete set of representatives for the $O(\Lambda)$-conjugacy classes of finite order isometries with negative definite coinvariant lattice. For each such representative $f\in O^+(\Lambda)$, we determine whether it is symplectic effective by checking the condition $\Lambda_f\cap \mathcal{W}(\Lambda)^{pex}$ (see \Cref{lemme symplectic}). In that way, we obtain 32 symplectic effective isometries, up to conjugacy, and we determine which one are regular symplectic effective by testing whether their associated coinvariant lattice contains vectors of $\mathcal{W}(\Lambda)$. 
\end{proof}%
The computations needed for the proof of \Cref{thm classification} took approximately a week, with a major part of the time dedicated to determining even order isometries (due to $D_\Lambda\cong (\mathbb{Z}/2\mathbb{Z})^8$).%

\begin{nota} We fix the following lattice notation for the display in \Cref{tab: table symplectic}%
\[ V \coloneqq {\scriptstyle{\begin{pmatrix}%
        0&1\\1&1%
    \end{pmatrix}}},\quad K_p \coloneqq {\scriptstyle{\begin{pmatrix}%
        (p+1)/2&-1\\-1&2%
    \end{pmatrix}}},\quad H_p \coloneqq {\scriptstyle{\begin{pmatrix}%
        (p-1)/2&1\\1&-2%
    \end{pmatrix}}}. \]%
Moreover, the lattices $L_{13}$, $L_{15}$ and $L_{29}$ are abstract even definite lattices in the prescribed genera (in column ``$g(\Lambda_f)$'').%
\end{nota}%

\begin{rmk}\label{rmk unique entries}%
    By inspecting \Cref{tab: table symplectic}, we observe that the invariant lattices $\Lambda^f$
associated with symplectic effective isometries are pairwise non-isometric.
It follows that the conjugacy class of a symplectic effective isometry
$f \in O^+(\Lambda)$ is uniquely determined by the isometry class of its invariant lattice $\Lambda^f$.
Similarly, among regular symplectic effective isometries, the coinvariant lattices $\Lambda_f$
are pairwise non-isometric, so the conjugacy class of such an isometry is uniquely determined
by the isometry class of $\Lambda_f$.
\end{rmk}%

The previous remark motivates the following definitions.

 \begin{defi}\label{defiinvolution}%
 \hfill%
 \begin{enumerate}%
     \item Let $\iota\in O^+(\Lambda)$ be an involution. We say $\iota$ is \emph{exceptional} if $\Lambda_{\iota}\simeq A_1$.%
     \item Let $X$ be an orbifold of Nikulin-type. An involution $\iota$ on $X$ is called \emph{exceptional} if so is $\rho_X(\iota)\in O^+(\HH^2(X,\Z))$.%
 \end{enumerate}%
 \end{defi}%
 \begin{rmk}\label{standardremark}%
 According to \Cref{thm classification} and \Cref{tab: table symplectic}, the lattice $\Lambda$ has a unique conjugacy class of exceptional involutions. Note that such an involution is described in \cite[Section 4.2]{Ulrike2}. If $X$ is a Nikulin orbifold constructed from a Hilbert scheme of two points on a K3 surface, we can construct an exceptional involution on $X$; this involution exchanges the two exceptional divisors $\delta'$ and $\Sigma'$ (see \Cref{nota} for the notation).%
\end{rmk}%

\subsection{Characterization of deformation equivalent automorphisms}
In this section we show that the lattice action is enough to characterize the deformation class of an automorphism. For what follows, we let $T$ denote an abstract deformation type of irreducible symplectic orbifolds, and $\Lambda_T$ is the associated lattice.
\begin{prop}\label{mainstandardlemma}%
Let $(X,\eta_X)$ and $(Y,\eta_Y)$ be two $\Lambda_T$-marked irreducible symplectic orbifolds of deformation type $T$. Assume that $[(X,\eta_X)]$ and $[(Y,\eta_Y)]$ lie in the same connected component $\mathcal{M}_{\Lambda_T}^o$ of the moduli space of $\Lambda_T$-marked irreducible symplectic orbifolds of deformation type $T$. 
Let $f \in \Aut_s(X)$ and $g \in \Aut_s(Y)$ be two symplectic automorphisms of finite order. 
Assume that $\rho_Y$ is injective.%

Then the pairs $(X,f)$ and $(Y,g)$ are deformation equivalent if and only if $\eta_X \rho_X(f) \eta_X^{-1}$ is $\Mon^2(\Lambda_T,\eta_X)$-conjugate to $\eta_Y \rho_Y(g) \eta_Y^{-1}$.%
\end{prop}%
\begin{proof}%
The idea of the proof is to apply \cite[Lemma 2.12]{Ulrike2}. For simplicity, let us denote again $f \coloneqq \rho_X(f)$ and $g \coloneqq \rho_Y(g)$. 
We denote by $f',g'\in \Mon^2(\Lambda_T)$ the isometries induced by $f$ and $g$ respectively through the markings $\eta_X$ and $\eta_Y$. By assumption there exists a monodromy ${h'}\in \Mon^2(\Lambda_T,\eta_X)$ such that $h'f' = g'{h'}$. Let us denote by $h\coloneqq \eta_Y^{-1}h'\eta_X\colon \HH^2(X, \Z)\to \HH^2(Y, \Z)$ the associated isometry such that $hf = gh$. 
The isometry $\eta_X\circ h^{-1}$ is a marking for $Y$ and $[(Y,\eta_X\circ h^{-1})]$ lies in $\mathcal{M}_{\Lambda_T}^o$.%

We set $\Lambda'\coloneqq \eta_X(\HH^2(X,\Z)^f)$; it is a lattice of signature $ (3,b'-3)$ with $b'=\rk\Lambda'$. We consider the inclusion of period domains $\mathcal{D}_{\Lambda'}\subset \mathcal{D}_{\Lambda_T}$. Let $[(\widetilde{X},\widetilde{\eta})]\in\mathcal{M}_{\Lambda_T}^o\cap\mathscr{P}_{\Lambda_T}^{-1}(\mathcal{D}_{\Lambda'})$ be such that $\widetilde{\eta}(\Pic(\widetilde{X}))^{\bot}=\Lambda'$. By definition, we have that $\widetilde{\eta}(\Pic(\widetilde{X})) = \eta_X(\HH^2(X, \Z)_f)$ so the lattice $\Pic(\widetilde{X})$ cannot contain any wall divisor. It follows from \cite[Corollary 4.14]{Ulrike0} that the positive cone $\mathcal{C}_{\widetilde{X}}$ is equal to the K\"ahler cone $\mathcal{K}_{\widetilde{X}}$. 
So according to \cite[Lemma 2.12]{Ulrike2} $[(X,\eta_X)]$ and $[(Y,\eta_X\circ h^{-1})]$ can be connected by a sequence of generic twistor spaces (see \Cref{Twistordef}) whose image under the period domain is contained in $\mathcal{D}_{\Lambda'}$. Then, as explained in \cite[Remark 2.11]{Ulrike2} the automorphism $f$ induces an automorphism on all these twistor spaces. In particular, it induces an automorphism $\hat{g}$ on $Y$ and by construction $\rho_Y(\hat{g})=h\circ f\circ {h}^{-1}=g$. Since $\rho_Y$ is injective, $g=\hat g$. This shows that $(X,f)$ and $(Y,g)$ are deformation equivalent.%
 \end{proof}%
\begin{cor}\label{latticecara}
 Let $X$ and $Y$ be two orbifolds of Nikulin-type, and $f \in \Aut_s(X)$ and $g \in \Aut_s(Y)$ be two symplectic automorphisms of finite order. The pairs $(X,f)$ and $(Y,g)$ are deformation equivalent, if and only if the invariant lattices $\HH^2(X, \Z)^{\rho_X(f)}$ and $\HH^2(Y, \Z)^{\rho_Y(g)}$ are isometric, 
 if and only if the coinvariant lattices  $\HH^2(X, \Z)_{\rho_X(f)}$ and $\HH^2(Y, \Z)_{\rho_Y(g)}$ are isometric.%
\end{cor}%
\begin{proof}%
By \cite[Proposition 8.1]{Ulrike2} $\rho_X$ is injective.
We can find markings such that $[(X,\eta_X)]$ and $[(Y,\eta_Y)]$ lie in the same connected component of the moduli space.
Then, \Cref{mainstandardlemma} reduces the deformation equivalence of symplectic automorphisms to
monodromy conjugacy of their actions on $\Lambda$. Moreover,
\Cref{thm classification} and \Cref{rmk unique entries} tell us that the corresponding conjugacy classes in $\Mon^2(\Lambda,\eta_X)=\Mon^2(\Lambda)$ are determined by the invariant lattice, as well as by the coinvariant lattices.%
\end{proof}%
\subsection{Classification of standard symplectic automorphisms}%
The objective of this section is to determine which symplectic automorphisms are standard.%
\begin{defi}[Natural and standard automorphisms]%
\hfill%
\begin{enumerate}%
    \item  Let $Y$ be an irreducible holomorphic symplectic manifold of $\textnormal{K3}^{[2]}$-type endowed with a symplectic involution $i$. 
 Let $M'$ be the Nikulin orbifold constructed from $(Y,i)$ as in \Cref{Niku}. Let $g_0\in \Aut(Y)$ be such that $g_0$ commutes with $i$. 
 Then $g_0$ induces an automorphism $g$ on $M'$. The automorphism $g$ is called a \emph{natural automorphism} on $M'$ and the pair $(M',g)$ is said to be a \emph{natural pair}.%
 \item If $Y$ in point 1. is actually of the form $S^{[2]}$ where $S$ is a \textnormal{K3} surface, and $g_0\in \Aut(Y)$ is induced from an action on $S$, we call the automorphism $g\in \Aut(M')$ \emph{K3-natural}.%
 \item Let $X$ be an irreducible symplectic orbifold of Nikulin-type and let $f\in\Aut(X)$. The automorphism $f$ is said \emph{(K3-)standard} if the pair $(X,f)$ is deformation equivalent to a (K3-)natural pair. In this case, we say the pair $(X,f)$ is \emph{(K3-)standard}.%
\end{enumerate}%
\end{defi}%

We summarize the notation of the previous definition in the following diagram:%
\begin{equation}%
\xymatrix{ & Y\ar[d]^{\pi}\ar@(ur,dr)[]^{g_0}\\%
\ar@(ul,dl)[]_{g}M'\ar[r]^r & Y/i\ar@(ur,dr)[]^{g}.}%
\end{equation}%
\begin{rmk}\label{naturalmorphism}%
By definition, $M'$ admits a natural automorphism $f$ of finite order $m\geq 1$ if the associated $Y$ admits an automorphism $f_0$, commuting with $i$, and $f$ is induced by $f_0$. By considering the group $G_0 \coloneqq \langle f_0, i\rangle$, we may reformulate the previous equivalently by saying that $Y$ admits an automorphism group $G_0$ with $i\in G_0$ a central symplectic involution so that $G_0/\left\langle i\right\rangle $ is cyclic of order $m$.%
\end{rmk}%

 \begin{rmk}\label{latticecararemark}
\Cref{latticecara} shows that being standard for a finite order symplectic automorphism is completely characterized by its action on the second cohomology lattice. To be more precise, let $(X,f)$ be an orbifold of Nikulin-type endowed with a finite order symplectic automorphism; the action of $f$ on $\Lambda$ is described by one of the entries of \Cref{tab: table symplectic}. If we can find a natural automorphism with the same action on $\Lambda$ (modulo conjugation) then $f$ is standard.%
 \end{rmk}%

\begin{defi}\label{defistandard}%
 Let $f\in O^+(\Lambda)$ be a finite order isometry. We say $f$ is \emph{standard} if there exists a standard pair $(X, g)$ and a marking $\eta$ of $X$ such that $f$ is conjugate to $\eta^{-1}\rho_X(g)\eta$.%
\end{defi}%

Using the idea explained in  \Cref{naturalmorphism}, we prove the following theorem.

\begin{thm}\label{prop type of symp bir}%
Let $(X,f)$ be an orbifold of Nikulin-type endowed with a symplectic birational automorphism of finite order. Then exactly one of the following holds:%
\begin{enumerate}[(i)]%
    \item $f$ is not regular,%
    \item $(X, f)$ is a standard pair,%
    \item there exists a marking $\eta$ of $X$ such that $\eta\rho_X(f)\eta^{-1}$ is of the form $h\circ\iota$ where $h$ is standard and $\iota$ is an exceptional involution,%
    \item the order of $f$ is divisible by 5,%
    \item $f$ is an involution such that $\HH^2(X, \mathbb{Z})_{\rho_X(f)}\simeq D_{10}(2)$.%
\end{enumerate}%
\end{thm}%
\begin{proof}%
From the information available in the first eight columns of \Cref{tab: table symplectic}, we only need to determine which are the standard isometries of $\Lambda$. 
According to \Cref{mainstandardlemma} and \Cref{standardremark}, standard isometries of $\Lambda$ are conjugate to the ones realized by a natural automorphisms. It therefore suffices to determine which entries of \Cref{tab: table symplectic} can be realized by natural symplectic automorphisms on Nikulin orbifolds.%

Let $M'$ be a Nikulin orbifold obtained from $(Y,i)$ with $Y$ a manifold of $\textnormal{K3}^{[2]}$-type and $i$ a symplectic involution on $Y$. 
As explained in \Cref{naturalmorphism}, to obtain a natural symplectic automorphism $f$ of order $r$ on the Nikulin orbifold $M'$, we need a symplectic group $G_0$ on $Y$ with $i\in G_0$ a central involution such that 
$G_0/\left\langle i\right\rangle $ is a cyclic group of order $r$. Note moreover that in this case $\rk \HH^2(M',\Z)^f=\rk \HH^2(Y,\Z)^{G_0}+1$; indeed, there is the exceptional divisor $\Sigma'$ which is always fixed by a natural automorphism (see \Cref{nota} for the notation). We can look at \cite[Table 12]{hohn} to list all such possible groups $G_0$ (note that the groups of symplectic automorphisms on an IHS manifold of $\textnormal{K3}^{[2]}$-type are the ones which are not with a ``--'' in the column ``Type'', see also \cite[Theorem 8.7]{hohn}). We list all the possibilities in the following table.%
\begin{center}%
{
\setlength{\tabcolsep}{5pt}%
\renewcommand%
\arraystretch{1.3}%
\rowcolors{1}{white}{lightgray!40!white}%
 \begin{tabular}{ccc|ccc|ccc}%
   \hline%
    $r$&$G_0$&$\rk \HH^2(M',\Z)^f$&$r$&$G_0$&$\rk \HH^2(M',\Z)^f$&$r$&$G_0$&$\rk \HH^2(M',\Z)^f$\\%

    \hline%

    $2$&$C_2\times C_2$&12&$3$&$C_6$&6&$6$&$C_{12}$&4\\%

$2$&$C_4$&10&$4$&$C_8$&6&$6$&$C_{2}\times C_6$&6\\%

$3$&$C_6$&8&$4$&$C_2\times C_4$&8&$7$&$C_{14}$&4\\%

    \hline%
\end{tabular}%
}
\end{center}%
Then comparing with the data computed for the proof of \Cref{thm classification}, and reported in \Cref{tab: table symplectic}, we obtain our result.%

\end{proof}%

All types of regular symplectic involutions on Nikulin-type orbifolds were already known before this work, at the exception of the one described in entry no. 7 of \Cref{tab: table symplectic}: The exceptional involutions (entry no. 2 of \Cref{tab: table symplectic}) are described in \cite[Section 4.2]{Ulrike2}, and the other four (entries no. 3--6) were determined and studied by Piroddi in her PhD thesis (see for instance \cite{piroddi}). Note that in \cite[Remark 2.3.7]{piroddi}, Piroddi mentioned that knowing the invariant and coinvariant lattices of a symplectic involution was possibly not enough to conclude that it is standard. A consequence of \Cref{mainstandardlemma} and \Cref{prop type of symp bir} is that knowing the isometry class of the invariant lattice of such an involution is actually enough to say whether it is standard.

\begin{rmk}%
    The two isometries of type (iv) from \Cref{prop type of symp bir} are actually related by an exceptional involution, meaning that entry no. 31 of \Cref{tab: table symplectic} can be obtained from entry no.~20 by composing with an exceptional involution. The only difference with the isometries of type (iii) is that entry no. 20 is not standard.%
\end{rmk}%

\subsection{Further results and comments}

\subsubsection*{Finite groups of symplectic birational automorphisms}

It follows from the proof of \Cref{lem:OD_Lambda_generated_by_reflections} (see also \cite[Proposition 2.6]{brandhorst-veniani}) that $O(D_{\Lambda})$ is isomorphic to the group $C_2\times O_7(2)$, where $O_7(2)\cong O(K)\cong \text{Sp}(K/R)$. The $C_2$ factor is generated by the isometry exchanging the two copies of $A_1$ in $\Lambda$, and $O_7(2)$ is a finite simple group of order $1451520 = 2^9\times 3^4\times 5\times 7$.%

\begin{rmk}\label{rmk: image disc}%
    By analyzing the data computed for the proof of \Cref{thm classification}, we observe that for each $f\leq O^+(\Lambda)$ as described in \Cref{tab: table symplectic}, the image $D_f$ of $f$ along the map $O^+(\Lambda)\to O(D_\Lambda)$ lies in the normal subgroup $O_7(2)\trianglelefteq O(D_\Lambda)$. Similarly, as reported in the column labeled ``$\textnormal{ord}(D_f)$'', the only nontrivial symplectic effective isometries of finite order acting trivially on $D_\Lambda$ are the exceptional involutions.%
\end{rmk}%

\begin{prop}\label{lemme action of disc}%
    Let $X$ be an orbifold of Nikulin-type, let $G\leq \textnormal{Bir}_s(X)$ be a finite subgroup and let $G^\sharp$ be the kernel of the map $G\to O(D_{\HH^2(X, \Z)})$. Then $\#G^\sharp \leq 2$ and $G/G^\sharp$ embeds into $O_7(2)$.%
\end{prop}%

\begin{proof}%
    Fix a marking $\eta\colon \HH^2(X, \mathbb{Z})\to \Lambda$ of $X$ and, by abuse of notation, let us denote again $G \coloneqq \eta\rho_X(G)\eta^{-1}\leq O^+(\Lambda)$. Let $\pi\colon G\to O(D_\Lambda)$ be the discriminant representation of $G$, with kernel $G^\sharp \coloneqq \ker\pi$. Let $f\in G^\sharp$: According to \Cref{rmk: image disc}, either $f$ is the identity, or it is given by the reflection along a vector of square $-2$ and divisibility 2 in $\Lambda$. There are two cases. First, if $G^\sharp = \{\id\}$, then $G$ embeds into $O_7(2)$ according to \Cref{rmk: image disc}. Second, if $G^\sharp$ is nontrivial, then there exists at least one $f$ as before. Suppose that there exist two such isometries $f,g\in G^\sharp$ which are given by the reflections along two vectors of square $-2$ and divisibility 2 in $\Lambda$. The composition $f\circ g\in G^\sharp$ is thus either the identity or a reflection. But the composition of two reflections is not a reflection, and we deduce that $f=g$. Hence $G^\sharp$ is cyclic of order 2, generated by an isometry $f$ which is $O^+(\Lambda)$-conjugate to the exceptional involution $\iota$. Following \Cref{rmk: image disc}, we then obtain that $G/G^\sharp$ embeds into $O_7(2)$.%
\end{proof}%

Note that this differs with the behavior of finite groups of symplectic birational automorphisms on the known IHS manifolds. In fact, as reported in \cite[Lemma 6.15]{mul25} and the references therein, if $X$ is a known example of IHS manifolds and $G\leq\Bir(X)$ is a finite subgroup of birational automorphisms, then $[G:G^\sharp]\leq 2$ holds. Moreover, if $G$ is regular and symplectic, then $G = G^\sharp$ (see \cite[Theorem~26]{mon16b}, \cite[Lemma~5.1]{mtw18}, \cite[Theorem~1.1]{gov23}, \cite[Theorem~1.1]{ggov24}). It follows that the methods used to classify finite groups of symplectic birational automorphisms on the known IHS manifolds do not apply to the case of Nikulin-type orbifolds. We refer to \cite{mon16b}, \cite[\S 5]{mtw18}, \cite{ggov24} and \cite{mm23} for an overview of this problem in the case of IHS manifolds. See also \cite{mon13} for an early description of the general approach to this problem, and \cite[\S 7.2 and 9.2]{mul25} for a more recent summary and generalization of the aforementioned methods.

\begin{cor}%
 Let $(X, f)$ be an orbifold of Nikulin-type endowed with a symplectic birational automorphism of finite order $m$. 
 We set $D_{f}$ for the action on $D_{\HH^2(X, \Z)}$ induced by $f$. If $D_{f}$ has order less than $m$, then $f^{m/2}$ is an exceptional involution on $X$.%
\end{cor}%

\begin{proof}%
    Let $G \coloneqq \langle f\rangle$ be the cyclic group generated by $f$: By our assumption on $D_f$, we have that $G^\sharp$ is nontrivial, and \Cref{lemme action of disc} tells that $G^\sharp$ has order 2. This implies that $m$ is even, and in fact, $G^\sharp$ is generated by $f^{m/2}$. It then follows from our classification that $f^{m/2}$, being an involution acting trivially on $D_{\HH^2(X, \Z)}$, is exceptional (see again  \Cref{rmk: image disc}).
\end{proof}%

\subsubsection*{Nonsymplectic automorphisms of prime order}
The algorithm from \cite{brandhorst-hofmann} can also be applied to compute a complete set of representatives for the isometry classes of effective prime order isometries of $\Lambda$ which are not symplectic. Similarly to what was done in the symplectic case, we obtain the following.%

\begin{prop}\label{prop class nonsymplectic}%
    Up to $\Mon^2(\Lambda)$-conjugation, there exist in $\Mon^2(\Lambda)$ exactly:%
    \begin{enumerate}[(i)]%
        \item 2226 nonsymplectic effective groups of order 2;%
        \item 39 nonsymplectic effective groups of order 3;%
        \item 6 nonsymplectic effective groups of order 5;%
        \item 4 nonsymplectic effective groups of order 7.%
    \end{enumerate}%
    Representatives of the conjugacy classes in terms of matrices are found in the folder "nonsymplectic" in the database \cite{database}.%
\end{prop}%

\begin{proof}%
    Let $p$ be a prime number, and fix a primitive $p$th root of unity $\zeta_p$. 
    We proceed similarly as in the proofs of \Cref{lemme symplectic} and \Cref{thm classification}. However, this time, following \cite[Theorem 3.13]{bc22}, we classify isometries $f\in \Mon^2(\Lambda)$ whose invariant lattice has signature $(1, \ast)$ and such that the real quadratic space%
    \[\ker(f_{\mathbb{R}}+f_{\mathbb{R}}^{-1}-\zeta_p-\zeta_p^{-1})\]%
    has signature $(2, \ast)$. Indeed, the previous conditions ensure that we can find a general enough eigenvector $\omega\in\Lambda_f\otimes\mathbb{C}$ of $f_{\mathbb{C}}$, associated to the eigenvalue $\zeta_p$, satisfying 
    \[(\mathbb{R}\Rea(\omega)\oplus\mathbb{R}\Ima(\omega))^\perp\cap \Lambda = \Lambda^f.\]
    By the surjectivity of the period map $\mathscr{P}_{\Lambda}$ (\Cref{surjectivity period map}), there exists a $\Lambda$-marked Nikulin-type orbifold $(X, \eta)$ with period $\mathbb{C}\omega$. By what we have described above, the pair $(X, \eta)$ verifies $\eta(\textnormal{NS}(X)) = \Lambda^f$ and $\eta(\textnormal{T}(X)) = \Lambda_f$. It follows that the isometry $f_X \coloneqq \eta^{-1}f\eta$ is a Hodge monodromy (by the choice of $\omega$) and it fixes pointwise the N\'eron--Severi lattice of $X$. By the Hodge version of the global Torelli theorem for irreducible symplectic orbifolds \cite[Theorem 1.1]{Ulrike0}, the isometry $f_X$ is induced by a nonsymplectic automorphism on $X$, and $f$ is therefore effective. Note that this time, there are no extra numerical conditions as in \Cref{lemme symplectic}, because of the trivial action of $f_X$ on $\textnormal{NS}(X)$.
    
\end{proof}%
The computations needed for the proof of \Cref{prop class nonsymplectic} were realized on a span of 2 months, and required additional tweaks on a case-by-case basis for the involutions.

\newpage
\subsection{Table of results}
For the entries in \Cref{tab: table symplectic} corresponding to case (ii) in \Cref{prop type of symp bir}, we let the last column be ``K3'', resp. ``$\text{K3}^{[2]}$'', if the pair is K3-standard, resp. standard but not K3-standard. For the entries corresponding to symplectic automorphisms of type (iii), we give the number ``no.'' of the associated standard isometry and we add ``$\circ\ \iota$''. For the remaining types (i), (iv) and (v), we leave the last column empty ``\textemdash''.%

The coinvariant lattices $\Lambda_f$ are listed if they are easy to describe. We let moreover $L_{13}$, $L_{15}$ and $L_{29}$ be abstract negative definite lattices: If two entries share the same $L_i$ in the column ``$\Lambda_f$'', it means that the associated coinvariant lattices are isometric. In the other cases, the entry is left blank and the associated coinvariant lattice can be found in the database \cite{database}.%

The notation $g(L)$, given a lattice $L$, refers to the symbol of the genus of $L$ following Conway--Sloane convention (see \cite[Chapter 15]{splg}).

\begin{center}
\captionof{table}{Conjugacy classes of symplectic isometries of $\Lambda$ (see \Cref{thm classification})}
{\footnotesize
\setlength{\tabcolsep}{1.5pt}
\renewcommand
\arraystretch{1.5}
\rowcolors{1}{white}{lightgray!40!white}
    \begin{tabular}{ccccccccc}\label{tab: table symplectic}
   
     no.&$\text{ord}(f)$&$\text{ord}(D_f)$&$\Lambda^f$&$g(\Lambda^f)$&$\Lambda_f$&$g(\Lambda_f)$&Regular&Type\\

    \hline
    1&1&1&$\Lambda$&$\II_{(3,13)}2^8_6$&\textemdash&\textemdash&True&K3\\

    2&2&1&$U^{\oplus3}\oplus D_8^\vee(2)\oplus A_1$&$\II_{(3,12)}2^7_7$&$A_1$&$\II_{(0,1)}2^1_7$&True&$\text{(no. 1)}\circ\iota$ \\

    3&2&2&$U^{\oplus2}\oplus U(2)\oplus D_4(2)\oplus A_1^{\oplus2}$&$\II_{(3,9)}2^6_24^2$&$D_4(2)$&$\II_{(0,4)}2^{-2}4^{-2}$&True& K3 \\

    4&2&2&$U^{\oplus2}\oplus U(2)\oplus D_4(2)\oplus A_1$&$\II_{(3,8)}2^{5}_34^{2}$&$D_4(2)\oplus A_1$&$\II_{(0,5)}2^{-3}_74^{-2}$&True& $\text{(no. 3)}\circ\iota$\\

    5&2&2&$U(2)^{\oplus3}\oplus A_1^{\oplus2}\oplus A_1(2)^{\oplus 2}$&$\II_{(3,7)}2^8_64^2_6$&$D_6(2)$&$\II_{(0,6)}2^{-4} 4^{-2}_6$&True& K3\\

    6&2&2&$U(2)^{\oplus2}\oplus V(2)\oplus A_1\oplus A_1(2)^{\oplus2}$&$\II_{(3,6)}2^7_74^2_6$&$D_6(2)\oplus A_1$&$\II_{(0,7)}2^{-5}_74^{-2}_6$&True&  $\text{(no. 5)}\circ\iota$\\

    7&2&2&$U(2)^{\oplus 2}\oplus V(4)$&$\II_{(3,3)}2^44^2_0$&$D_{10}(2)$&$\II_{(0, 10)}2^84^2_6$&True& \textemdash\\

    8&2&2&$U(2)\oplus V(2)\oplus U(4)\oplus A_1(2)^{\oplus2}$&$\II_{(3,5)}2^4_04^4_6$&\textemdash&$\II_{(0,8)}2^{-4}_64^{-4}_6$&False&\textemdash\\

    9&3&3&$V(6)^{\oplus2}\oplus A_1(-1)\oplus A_1(3)$&$\II_{(3,3)}2^{-6}_63^5$&\textemdash&$\II_{(0,10)}2^{-2}3^{-5}$&True&$\text{K3}^{[2]}$\\

    10&3&3&$U\oplus U(3)\oplus U(6)\oplus A_1^{\oplus2}$&$\II_{(3,5)}2^4_63^4$&\textemdash&$\II_{(0,8)}2^43^4$&True&K3\\

    11&4&4&$U\oplus U(4)\oplus V(4)\oplus A_1^{\oplus 2}$&$\II_{(3,5)}2^2_64^4_0$&\textemdash&$\II_{(0,8)}2^24^4_0$&True&K3\\

    12&4&4&$U\oplus U(4)\oplus V(4)\oplus A_1$&$\II_{(3,4)}2^1_74^4_0$&\textemdash&$\II_{(0,9)}2^3_74^4_0$&True& $\text{(no. 11)}\circ\iota$\\

    13&4&4&$U(4)\oplus A_1(-2)\oplus A_1(-4)\oplus A_1^{\oplus2}$&$\II_{(3,3)}2^2_64^3_18^1_1$&$L_{13}$&$\II_{(0,10)}2^24^3_78^1_7$&True&K3\\

    14&4&4&$U(4)\oplus A_1(-2)\oplus A_1(-4)\oplus A_1$&$\II_{(3,2)}2^1_74^3_18^1_1$&\textemdash&$\II_{(0, 11)}2^3_74^3_78^1_7$&True& $\text{(no. 13)}\circ\iota$\\

    15&4&4&$U\oplus U(4)^{\oplus 2}$&$\II_{(3,3)}4^4$&$L_{15}$&$\II_{(0, 10)}2^4_64^4$&False&\textemdash \\

    16&4&4&$V(2)\oplus U(4)^{\oplus2}$&$\II_{(3,3)}2^2_04^4$&$L_{15}$&$\II_{(0, 10)}2^4_64^4$&False&\textemdash \\

    17&4&4&$U(2)\oplus U(4)^{\oplus 2}$&$\II_{(3,3)}2^24^4$&$L_{15}$&$\II_{(0, 10)}2^4_64^4$&False&\textemdash \\

    18&4&4&$U(2)\oplus U(4)\oplus A_1(4)\oplus A_1(-2)$&$\II_{(3,3)}2^24^3_18^1_7$&$L_{13}$&$\II_{(0,10)}2^24^3_58^1_1$&False&\textemdash \\

    19&4&4&$U(4)\oplus A_1(-2)\oplus A_1(-4)$&$\II_{(3,1)}4^3_18^1_1$&\textemdash&$\II_{(0,12)}2^44^3_58^1_7$&False&\\

    20&5&5&$V(10)\oplus A_1(-1)\oplus A_1(-5)$&$\II_{(3,1)}2^{-4}_65^{-3}$&\textemdash&$\II_{(0,12)}2^{-4}5^{-3}$&True&\textemdash \\

    21&6&3&$U\oplus U(3)\oplus U(6)\oplus A_1$&$\II_{(3,4)}2^3_73^4$&\textemdash&$\II_{(0,9)}2^5_73^4$&True&$\text{(no. 10)}\circ\iota$\\

    22&6&6&$U(6)\oplus A_1(-1)^{\oplus2}\oplus A_1(3)\oplus A_1$&$\II_{(3,3)}2^{-6}_63^{-3}$&\textemdash&$\II_{(0,10)}2^{-6}3^3$&True&K3\\

    23&6&6&$U(6)\oplus A_1(-1)^{\oplus2}\oplus A_1(3)$&$\II_{(3,2)}2^{-5}_73^{-3}$&\textemdash&$\II_{(0,11)}2^{-7}_73^3$&True&$\text{(no. 22
    )}\circ\iota$\\

    24&6&6&$V(12)\oplus A_1(-1)\oplus A_1(-3)$&$\II_{(3,1)}2^2_64^{-2}_23^3$&\textemdash&$\II_{(0,12)}2^24^{-2}_63^{-3}$&True&$\textnormal{K3}^{[2]}$\\

    25&6&6&$A_1(2)\oplus A_1(-2)\oplus A_2(-2)$&$\II_{(3,1)}2^{-2}4^2_03^1$&\textemdash&$\II_{(0,12)}2^64^{-2}_23^{-1}$&False&\textemdash \\

    26&6&6&$V(-6)\oplus A_1(-2)^{\oplus 2}$&$\II_{(3,1)}2^2_04^2_23^{-2}$&\textemdash&$\II_{(0,12)}2^{-6}_44^{-2}_43^{-2}$&False&\textemdash \\

    27&7&7&$K_7\oplus H_7(2)$&$\II_{(3,1)}2^2_67^2$&\textemdash&$\II_{(0,12)}2^67^2$&True&$\textnormal{K3}^{[2]}$\\

    28&8&8&$U(8)\oplus A_1(-1)\oplus A_1(-2)$&$\II_{(3,1)}2^1_14^1_18^2$&\textemdash&$\II_{(0,12)}2^{-5}_34^{-1}_58^2$&False&\textemdash \\

    29&8&8&$U\oplus A_1(-4)^{\oplus 2}$&$\II_{(3,1)}8^2_2$&$L_{29}$&$\II_{(0,12)}2^{-6}_28^{-2}_2$&False&\textemdash \\

    30&8&8&$U(2)\oplus A_1(-4)^{\oplus 2}$&$\II_{(3,1)}2_2^28^2_2$&$L_{29}$&$\II_{(0, 12)}2_2^{-6}8^{-2}_2$&False&\textemdash \\

    31&10&5&$A_1(-5)^{\oplus3}$&$\II_{(3,0)}2^{-3}_75^{-3}$&\textemdash&$\II_{(0,13)}2^{-5}_75^{-3}$&True&\textemdash \\

    32&14&7&$K_7\oplus A_1(-7)$&$\II_{(3,0)}2^1_77^2$&\textemdash&$\II_{(0,13)}2^7_77^2$&True&$\text{(no. 27)}\circ\iota$\\

    \hline    
\end{tabular}
}
\end{center}

\newpage
\bibliographystyle{amssort}%
\noindent%
Gr\'egoire \textsc{Menet}%

\noindent%
Institut de Math\'ematiques de Bourgogne%

\noindent%
9 avenue Alain Savary,%

\noindent%
21078 Dijon Cedex%

\noindent%
{\tt gregoire.menet@ac-amiens.fr}\\%

\noindent%
Simon \textsc{Brandhorst}%

\noindent%
Fakult\"at f\"ur Mathematik und Informatik,%

\noindent%
Universit\"at des Saarlandes,%

\noindent%
Campus E2.4, 66123 Saarbr\"ucken, Germany%

\noindent%
{\tt brandhorst@math.uni-sb.de}\\%

\noindent%
Stevell \textsc{Muller}%

\noindent%
Institut f\"ur Algebraische Geometrie,%

\noindent%
Leibniz Universit\"at Hannover,%

\noindent%
Welfengarten 1, 30167 Hannover, Germany%

Previous address:

Fakult\"at f\"ur Mathematik und Informatik,

Universit\"at des Saarlandes,

Campus E2.4, 66123 Saarbr\"ucken, Germany

\noindent%
{\tt muller@math.uni-hannover.de}%
\end{document}